\documentclass[11pt]{amsart}
\usepackage{graphicx}
\usepackage{epstopdf}
\usepackage{url}
\usepackage{latexsym}
\usepackage{amsfonts,amsmath,amssymb}
\newtheorem{theorem}{Theorem}[section]
\newtheorem{proposition}[theorem]{Proposition}
\newtheorem{lemma}[theorem]{Lemma}

\newtheorem{corollary}[theorem]{Corollary}
\newtheorem{question}[theorem]{Question}
\newtheorem{example}[theorem]{Example}

\date{\today}

\begin{document}

\title[Unique longest increasing subsequences]{Pattern avoiding permutations and involutions with a unique longest increasing subsequence}

\author{Mikl\'os B\'ona \and Elijah DeJonge}
\address{Department of Mathematics, University of Florida, $358$ Little Hall, PO Box $118105$,
Gainesville, FL, $32611-8105$ (USA)}
\email{bona@ufl.edu}

\begin{abstract} We investigate permutations that avoid a pattern of length three and have a {\em unique}
longest increasing subsequence. 
\end{abstract}

\maketitle

\section{Introduction}
Let $p=p_1p_2\cdots p_n$ be a permutation. 
An {\em increasing subsequence} in $p$ is just a subset of entries $p_{j_1}<p_{j_2}< \cdots < p_{j_k}$ so that 
$j_1<j_2<\cdots <j_k$. Note that the positions $j_1,j_2,\cdots, j_k$ are {\em not} required to be consecutive. 

We say that $p$ has a {\em unique longest increasing subsequence}, or ULIS, if $p$ has an increasing subsequence that
is longer than all other increasing subsequences. For instance, 2314 has a ULIS, namely the sequence 234, but $p=246135$ does not,
since  246, 245, 235, and 135 are all increasing subsequences of maximal length in $p$. 

The number of all permutations of length $n$ that have a unique longest increasing subsequence may be quite difficult to determine. 
Sequence A167995 in the Online Encyclopedia of Integer Sequences \cite{oeis} contains these numbers for $n\leq 15$.

In this paper, we consider permutations that avoid a given pattern $q$ of length three that have a ULIS. 
 We say that a permutation $p$ {\em contains} the pattern $q=q_1q_2\cdots q_k$ 
if there is a $k$-element set of indices $i_1<i_2< \cdots <i_k$ so that $p_{i_r} < p_{i_s} $ if and only
if $q_r<q_s$.  If $p$ does not contain $q$, then we say that $p$ {\em avoids} $q$. For example, $p=3752416$ contains
$q=2413$, as the first, second, fourth, and seventh entries of $p$ form the subsequence 3726, which is order-isomorphic
to $q=2413$.  A recent survey on permutation 
patterns can be found in \cite{vatter} and a book on the subject is \cite{combperm}. The basic facts that we will simply call well-known are
explained in detail in these sources.

If $q$ is any given pattern of length three,
then it is well known that the number of {\em all} permutations of length $n$ that avoid $q$ is the Catalan number $C_n={2n\choose n}/(n+1)$. 
As structures counted by the Catalan numbers have been extensively studied, we find it surprising that the questions discussed in this paper 
prove to be both difficult and unexplored. The diverse nature of the results we prove will also be interesting as we will see that depending on the pattern $q$, 
the portion of $q$-avoiding permutations that have a ULIS may converge to a positive constant, converge to 0 at a subexponential speed, or converge to 
zero at an exponential speed.

 It is straightforward to see that
if $p=p_1p_2\cdots p_n$ has a ULIS, then so does its group-theoretical inverse $p^{-1}$, and so does the reverse complement $p^{revc}=((n+1-p_n)\: (n+1-p_{n-1}) \: \cdots \: (n+1-p_1))$ of $p$. Therefore, the six possible choices of $q$ can be reduced to four, namely 123, 132, 231, and 321, and we will see that one of them, 123, leads to a trivial situation.

We denote by $u_n(q)$ the number of permutations of length $n$ that avoid the pattern $q$ and have a ULIS, and by $u_q(z)$ the ordinary generating function
$u_q(z)=\sum_{n\geq 0}u_n(q)z^n$, with $u_0(q)=1$. Similarly, we denote by $i_n(q)$  the number of $q$-avoiding involutions of length $n$ which have a ULIS. 

\section{The pattern 231}
\subsection{Permutations}
Let $p=p_1p_2\cdots p_n$ be a permutation that avoids 231. 
Let $p_i=n$. Let $L=p_1p_2\cdots p_{i-1}$, and let $R=p_{i+1}\cdots p_n$. Note that all entries in $L$ must be smaller than
all entries of $R$.  If $p$ has a
ULIS, then it is necessary for both $L$ and $R$ to have a ULIS (since increasing sequences in $L$ or $R$ can be extended to increasing sequences in $p$). 
For this discussion, we consider empty strings as permutations having a ULIS. 

On the other hand, if $L$ and $R$ both have a ULIS, then the only case in which $p$ fails to have a ULIS is when $R$ consists of exactly one entry, (which is necessarily the entry $n-1$). 
This leads to the generating function identity
\[u_{231}(z)=u_{231}(z) z (u_{231}(z)-z) +1 .\]

This yields 
\begin{equation} \label{genf-231} u_{231}(z) = \frac{1+z^2-\sqrt{1-4z+2z^2+z^4}}{2z},\end{equation}
showing that the first few numbers $u_n(231)$, starting with $n=0$ are, 1, 1, 1, 2, 5, 13, 35, 97, 275, 794. 
This is sequence A082582 in the Online Encyclopedia of Integer Sequences \cite{oeis}. The dominant singularity of 
$U(z)$ is close to 0.2956, implying that the exponential growth rate of the sequence $u_n(231)$ is the reciprocal of that number, 
that is, close to 3.383. 

The following simple observation will be useful for us one more time in this paper, so we formally announce it. 
We say that a pattern is {\em indecomposable} if it cannot be cut into two parts so that every entry before the cut is smaller than every entry after the cut. 
For instance, $q=3142$ is indecomposable, but $q=21453$ is not as it can be cut between the entries 1 and 4. 
 
\begin{proposition} \label{fekete-lim}
Let $q$ be an indecomposable pattern. Then the limit $\lim_{n\rightarrow \infty} \left ( u_n(q) \right) ^{1/n} $ exists.
\end{proposition}

\begin{proof}
 For any positive $m$ and $n$, the inequality $u_{m}(q) u_{n}(q) \leq u_{m+n} (q)$ holds, 
since if $p$ and $r$ are   q-avoiding permutations with a ULIS, then so is $p\oplus r$, that is, the permutation of length $m+n$ that starts with $p$ on the entries $\{1,2,\cdots ,m\}$, 
then continues with a $r$-pattern taken on the entries $\{m+1,m+2,\cdots ,m+n\}$. Therefore, by Fekete's Lemma on superadditive sequences, the limit of the sequence
$\lim_{n\rightarrow \infty} \left ( u_n(q) \right) ^{1/n}$ exists. 
\end{proof}

We know that the limit in Proposition \ref{fekete-lim} is finite, since it is at most as large as the analogously defined limit for {\em all} $q$-avoiding permutations, and that is well-known to be finite. 

Proposition \ref{fekete-lim} and (\ref{genf-231}) together imply that $ \lim_{n\rightarrow \infty} \left ( u_n(231) \right) ^{1/n} \approx 3.383 $. We could have determined this
without Proposition \ref{fekete-lim}, because (\ref{genf-231}) provided an explicit formula for $u_{231}(z)$. In Section \ref{sec-321}, we will not have such a form for the relevant generating
function, but Proposition \ref{fekete-lim} will still enable us to compute corresponding limit. 

\subsection{Involutions} If $p$ is a 231-avoiding involution, then $p$ also avoids the inverse of 231, that is, the pattern 312. That means that $p$ is a {\em layered} permutation,
meaning that it consists of a series of decreasing subsequences of consecutive entries, (the layers), so that the entries decrease within each layer but increase among the layers, 
as in 32154876. The only way for such a permutation to have a ULIS is by having layers of length one only. That happens if and only if $p$ is the identity permutation, 
so $i_{n}(231)=1$ for all $n$. 

\section{The pattern 132}
\subsection{Permutations}
From our perspective, the pattern 132 behaves in surprising ways. This  leads to some intriguing questions, and even in the case when the answer to those questions is known, there is room for
simpler proofs. 

We say that a permutation $p$ is {\em skew indecomposable} if it is not possible to cut $p$ into two parts so that
each entry before the cut is larger than each entry after the cut. For instance, $p=3142$ is skew indecomposable,
but $r=346512$ is not as we can cut it into two parts by cutting between entries 5 and 1, to obtain $3465|12$.

 If $p$ is not skew indecomposable, then there is a unique way to cut $p$ into  nonempty skew indecomposable strings
$s_1,s_2,\cdots ,s_\ell$ of consecutive entries so that each entry of $s_i$ is larger than each entry of $s_j$ if $i<j$. We call these
strings $s_i$ the {\em skew blocks} of $p$. For instance, $p=67|435|2|1$ has four skew blocks, while skew indecomposable
permutations have one skew block. 

There is a well-known bijection $\psi$ from the set of all 132-avoiding permutations of length $n$ to the set of all plane rooted unlabeled trees on $n+1$ vertices.
If $p$ is a 132-avoiding, skew indecomposable permutation of length $n$, then it necessarily ends in $n$, so it is of the form $Ln$. We then define $\psi(p)$ to be the
plane tree whose root has only one child, and the subtree rooted at that child is precisely $\psi(L)$.  If $p$ is a skew decomposable 132-avoiding permutation of length $n$, then
$p=B_1B_2\cdots B_j$, where the $B_i$ are the skew blocks of $p$, each of which is necessarily indecomposable, and so ends with its largest entry. Therefore, for all $i$, the tree
 $\psi(B_i)$ is a tree
in which the root has only one child.  Then $\psi(p)$ is obtained by taking the sequence $\psi(B_1), \psi(B_2), \cdots \psi(B_j)$ in this order, and contracting the roots of these
$j$ trees into one vertex, which will be the root of $\psi(p)$. 

It is then easy to see that the number of longest increasing subsequences of $p$ is equal to the number of leaves in $\psi(p)$ that are at maximum distance from the root. Then the following
is known. 

\begin{theorem} \label{th-132} Let us select a rooted plane unlabeled tree on $n$ vertices uniformly at random, and let $a_{n,k}$ be the probability that the selected tree has $k$ leaves at maximum distance from the root. 
Then \[\lim_{n\rightarrow \infty} a_{n,k} =2^{-k}. \]
\end{theorem}

A very general theorem, that contains this result, is Theorem 2 in \cite{pittel}. An earlier, and more specific, reference is \cite{harris}. 

Setting $k=1$, we get the result that is equivalent to the question we were interested in. 
\begin{corollary} \label{132-1}
The equality \[\lim_{n\rightarrow \infty} \frac{u_n(132)}{C_n}=\frac{1}{2} \]
holds. 
\end{corollary}

As the proofs of Theorem \ref{th-132} are probabilistic, it is natural to ask the following question.
\begin{question} Is there a direct combinatorial proof of Corollary \ref{132-1}? \end{question}

The sequence of the numbers $u_n(132)$ is in the Online Encyclopedia of Integer sequences \cite{oeis}, as sequence A152880. Its first few entries, starting with $n=1$, are
1, 1, 3, 8, 23, 71, 229, 759, 2566. This numerical evidence raises two interesting questions. 

\begin{question} \label{qinj}  Is it true that $u_n(132) /C_n \geq 0.5$ for all $n$? In other words, is it true that for all $n$, there are at least as many 132-avoiding permutations of length $n$ 
that have a ULIS as 132-avoiding permutations of length $n$ that do not have a ULIS?
\end{question}

Needless to say, an injective proof of a positive answer to this question would be particularly interesting. 

\begin{question} \label{qmonotone} Is it true that the sequence $u_n(132)/C_n$ is monotone decreasing for $n\geq 3$? \end{question}

Note that an affirmative answer to Question \ref{qmonotone} would imply an affirmative answer to Question \ref{qinj}, because of the convergence result of Corollary \ref{132-1}.

\subsection{Involutions}
Consider the usual bijection $\phi$ from the set of 132-avoiding permutations of length $n$ to the set of Dyck paths of semilength $n$. The latter are lattice paths from $(0,0)$ to $(2n,0)$
consisting $n$ steps $(1,1)$, called U(p) steps, and $n$ steps $(1,-1)$, called D(own) steps that never go below the line $y=0$. The bijection $\phi$  is constructed inductively. If $p$ is skew decomposable, then $\phi(p)=\phi(B_1)\phi(B_2)\cdots \phi(B_k)$, where the $B_k$ are the skew blocks of $p$.  If $p$ is skew indecomposable, then $p$ is necessarily of the form $Ln$, and 
then we set $\phi(p)=U \phi (L) D$. 

A {\em peak}  in a Dyck path is a point that is immediately preceded by an U step and is immediately followed by a D step. The {\em height} of a peak is its $y$-coordinate. 

The following facts are straightforward to prove by induction. 
\begin{proposition}  \label{heights} Let $p$ be a 132-avoiding permutation. Then $p$ has a ULIS of length $k$ if and only if $\phi(p)$ has a unique peak of maximum height $k$.
\end{proposition} 

\begin{proposition} \label{involutions}
Let $p$ be a 132-avoiding permutation. Then $p$ is an involution if and only if $\phi(p)$ is symmetric to the vertical line $x=n$. 
\end{proposition}

Therefore, if $p$ is a 132-avoiding involution of length $n$ so that $\phi(p)$  has a unique peak $P$ of maximum height, then $P$ must be {\em on the vertical line} $x=n$, otherwise it would
not be unique. 

\begin{corollary} \label{paths} Let $n$ be a positive integer. Then 
the number of 132-avoiding involutions of length $n$ is equal to the number of lattice paths $\pi$ consisting of $n+1$ steps (each of which are U steps or D steps) so that each prefix and each
suffix of $\pi$ contains strictly more U steps than D steps.  
\end{corollary}

\begin{proof} If $p$ is a 132-avoiding involution of length $n$, just prepend the first half of  $\phi(p)$ with an U step. 
\end{proof} 

The lattice paths occurring in Corollary \ref{paths} have been studied in numerous papers, such as \cite{hackl,bousquet,zhao}. The most relevant enumerative result is proved in \cite{hackl}, where they are called {\em bidirectional
ballot sequences}, and the number of such paths of length $n$ is denoted by $B_n$. The result is that

\[\frac{B_n}{2^n}  \sim  \frac{1}{4n} + \frac{1}{6n^2} + O \left(\frac{1}{n^3} \right ) .\]
In particular, \[i_{n}(132)=B_{n+1} \sim \frac{2^{n+1}}{4(n+1)} = \frac{2^{n-1}}{n+1}.\]
On the other hand, it is well-known that the {\em total} number of 132-avoiding involutions of length $n$ is $I_n(132)={n \choose \lfloor n/2 \rfloor} \sim
 \sqrt{\frac{2}{\pi}} \frac{2^n}{\sqrt{n} }$.   Therefore, the probability that a 132-avoiding involution of length $n$ chosen uniformly at random has a ULIS is about $c / \sqrt{n}$. 

\section{The pattern 321} \label{sec-321}
\subsection{Permutations}
The case of $q=321$ is surprisingly difficult to handle. While we are not able to exhibit an explicit formula for the numbers $u_{321}(n)$, we are able to show that $\lim_{n\rightarrow \infty}
\left(u_n(321) \right)^{1/n} =4$, and that the generating function $u_{321}(z)$ is not rational. 

In what follows, we assume that the reader is familiar with the basic properties of the Robinson-Schensted (RS) correspondence from the set of all permutations of length $n$ to the set of
pairs of Standard Young Tableaux of the same shape on $n$ boxes. These basic properties, which we will call well-known, are discussed in detail in Bruce Sagan's book \cite{sagan}.  It is well-known that the RS bijection maps 321-avoiding permutations into pairs of SYT of the same shape that have
at most two rows. If, in addition, the permutation is of length $2m$, and has no ULIS, then its image is necessarily a pair of SYT of shape $(m,m)$, that is, consisting of two rows of $m$
boxes each. It is well-known that the number of Standard Young Tableaux of that shape is $C_m$, and this implies that the number of {\em pairs} of Standard Young Tableaux of that shape
is $C_m^2$. Note that with slight abuse of language, we will call the shape of the Standard Young Tableaux associated to $p$ by the RS correspondance the {\em shape of} $p$.

A {\em left-to-right maximum} in a permutation is an entry that is larger than all entries on its left, while a right-to-left minimum is an entry that is smaller than all entries on its right.
For instance, in 21354, the left-to-right maxima are 2, 3, and 5, while the right-to-left minima are 1, 3, and 4.

We are going prove a lower bound for the number of $u_{2m+1}(321)$. Our main tool is a bijection of Claesson and Kitaev \cite{claesson}. This bijection, which we will call $f$,  maps the set of all 321-avoiding permutations into the set of {\em indecomposable} 321-avoiding permutations as follows.

Let $p$ be a 321-avoiding permutation of length $n$. 
Locate all left-to-right maxima on the right of the entry 1 that are not right-to-left minima, and underline them. Insert a new maximum entry $n+1$ immediately to the left of the entry 1, and underline it. Finally, cyclically translate the set of underline entries one notch to the left so that $\underline{n+1}$ becomes the rightmost underlined entry to get $f(p)$. 

\begin{example} If $p=35124786$, then we underline 7 and 8, and insert 9, to get $35\underline{9}124\underline{7}\underline{8}6$. After a cyclical translation of the underlined entries to the left, we get
$f(p)=35\underline{7}124\underline{8}\underline{9}6$. \end{example}

\begin{lemma} The map $f$ described above is a bijection from the set of all 321-avoiding permutations of length $n$ to the set of all indecomposable 321-avoiding permutations of length $n+1$. \end{lemma}

\begin{proof}
First, let us assume that $f(p)$ is decomposable. That means that there exists an $i$ so that the rightmost $i$ entries of $f(p)$ are the largest $i$ entries of $f(p)$, that is, 
the entries $\{n-i+2, n-i+3,\cdots , n+1\}$. This is clearly not possible if the entry $n+1$ does not get moved into one of these last $i$ positions, so we can assume that it does. 
However, in that case, one of the last $i$ entries of $p$ gets moved {\em out} of one of the last $i$ positions, and that entry $x$ is a left-to-right maximum that is {\em not} a right-to-left
minimum. So, there is a right-to-left minimum $y$ on the right of $x$ that stays in one of the last $i$ positions. This contradicts the assumption that the rightmost $i$ positions of $f(p)$
contain the $i$ largest entries of $f(p)$. 

In order to prove that $f$ is a bijection, it suffices to show that it has an inverse. And that is easy, since the set of underlined entries is easy to recover from $f(p)$, by simply taking
the entry immediately on the left of the entry 1, then moving to the right and selecting all left-to-right maxima that are not right-to-left minima. 
\end{proof}

   The {\em rank} of an entry is the length of the longest increasing subsequence
ending in that entry. It follows that entries of the same rank form a decreasing subsequence. Therefore, in a 321-avoiding permutation, there are at most two entries of each rank. 

If the permutation $p$ has shape $(m,m)$, then each entry has rank at most $m$, and  there are no decreasing subsequences of length three or more, so $p$ has {\em exactly}  two entries of each rank $i$. Let us denote the leftmost entry of rank $i$ by $A_i$ and the rightmost entry of rank $i$ by $B_i$.
Then $A_i$ is a left-to-right maximum, and $B_i$ is a right-to-left minimum. For shortness, we will refer to the $A_i$ as the {\em large entries} and the $B_i$ as the {\em small entries}. 

\begin{example} If $p=351624$, then $A_1=3$, $B_1=1$, $A_2=5$, $B_2=2$, $A_3=6$, and $B_3=4$. \end{example}

The following property of the map $f$ is crucial for us. 

\begin{proposition} \label{notgrow}
Let $p$ be a permutation of shape $(m,m)$. Then $f$ does not increase the rank of any entry. That is, the rank of the entry $j$ in $f(p)$ is at most as large as the rank of the entry $j$
in $p$. 
\end{proposition}

\begin{proof} 
First, let us prove the statement for large entries. If $A_i$ is a large entry in $p$, then $f(p)$ either keeps $A_i$ fixed, or moves $A_i$ to the left. So in either case,  no entry leapfrogs $A_i$. So, the set
of entries that are on the left of $A_i$ that are smaller than $A_i$ does not acquire any new elements, therefore all increasing subsequences in $f(p)$ that end in $A_i$ are also increasing
subsequences in $p$ that end in $A_i$. 

Now let us prove the statement for small entries. Let us assume that $B_i$ is a counterexample to our claim with $i$ minimal. That means that in $f(p)$,  there is an increasing subsequence longer
than $i$ that ends in $B_i$.  As the small entries are not displaced by $f$, that subsequence must contain a large entry; let $A_j$ be the largest such entry in that subsequence. That means that
$A_j$ is the last element of that subequence before $B_i$.  Then $A_j$ must have rank at least $i$, since $B_i$ supposedly has rank larger than $i$. As the rank of
$A_j$ is not increased by $f$, that means that $A_j$ has rank $i$ or more in $p$. That is a contradiction, since $A_i>B_i$, so $A_j>B_i$, and so $A_j$ cannot be part of an increasing subsequence ending in $B_i$. 
\end{proof}


\begin{lemma} \label{injective}
Let $p$ be a permutation of shape $(m,m)$. Then $f(p)$ has a ULIS. 
\end{lemma}

\begin{proof}
Note that the map $f$ does not change the set of right-to-left minima of $p$, so $f(p)$ has $m$ right-to-left minima. However, $f(p)$ has $m+1$ left-to-right maxima, namely the $m$
left-to-right maxima of $p$, and the new maximum  entry $2m+1$. So, the left-to-right maxima form an increasing subsequence of length $m+1$ in $f(p)$. We will prove that this sequence is
the ULIS of $p$. 

Let us assume the contrary, that is, that there is an increasing subsequence of length $m+1$ 
in $f(p)$ that contains at least one small entry. Let $B_i$ be the rightmost small entry in such a subsequence. Note that this $B_i$ is necessarily (weakly) on the right of 1 as 
$B_1=1$ is the leftmost small entry. By Proposition \ref{notgrow}, the rank of $B_i$ in $f(p)$ is at most $i$. On the other hand, any increasing subsequence of large entries 
that are on the right of $B_i$ in $f(p)$ corresponds to such a subsequence in $p$, and therefore, is of length at most $m-i$, since the largest rank of any entry in $p$ is $m$. (Note that both in $p$ and $f(p)$, large entries on the right of
$B_i$ are automatically larger than $B_i$.) This gives the chosen sequence in $f(p)$ is of length at most $m$.
\end{proof}

As $f$ is injective, Lemma \ref{injective} shows that the number of permutations of shape $(m+1,m)$ that have a ULIS is at least the number of permutations of shape $(m,m)$, that is, 
\begin{equation} \label{lowerbound} C_m^2 = \frac{{2m \choose m}^2}{(m+1)^2} \sim \frac{4^{2m}}{m^3\pi} \sim \frac{2}{\pi} \frac{4^n}{n^3} ,\end{equation}
where $n=2m+1$. 
Here we used Stirling's formula that states that $m! \sim \left (\frac{m}{e} \right )^m \sqrt{2m\pi}$.

\begin{theorem} \label{321-size} The equality 
\[\lim_{n\rightarrow \infty} \left (u_{n}(321) \right) ^{1/n}  = 4 \]
holds. 
\end{theorem}

\begin{proof} Proposition \ref{fekete-lim} shows that the limit $L$ in the statement of the theorem exists.  On the other hand, (\ref{lowerbound})
proves that $4\leq L$, while the fact that $u_{n}(321)\leq C_n$ proves that $L\leq 4$. 
\end{proof}

Even though we do not know the explicit form of $u_{321}(z)$, we can use the result of Theorem \ref{321-size} and a method recently used in \cite{bona-nonrat} to prove that $u_{321}(z)$ is
not rational. 

\begin{theorem} The generating function $u_{321}(z)$ is not rational. \end{theorem}

\begin{proof} 
A 321-avoiding permutation has a ULIS if all its indecomposable blocks have a ULIS. Let $u_{n,1}(321)$ be the number of indecomposable 321-avoiding permutations with a ULIS, and let
$u_{321,1}(z)=\sum_{n\geq 1} u_{n,1}(321) z^n$. Then the equality 
\[u_{321}(z) = \frac{1}{1-u_{321,1}(z)} \] holds. Let us assume that $u_{321}(z)$ is rational, then so is $u_{321,1}(z)$. Both generating functions have their dominant singularity
at $z=1/4$, and, as they are both rational functions, those singularities are poles, implying that  at the pole $z=1/4$, the value of both generating functions goes to infinity. 
However, as all the coefficients of $u_{321,1}(z)$ are nonnegative, that implies that there is a $z_0\in (0,1/4)$ so that $u_{321,1}(z_0)=1$. That means that $u_{321}(z)$ has a singular
point at $z_0<1/4$, which is a contradiction.
\end{proof}

\subsection{Involutions}

A theorem of Sch\"utzenberger \cite{beissinger, schutz} says that if $p$ is an involution, then the number of fixed points of $p$ is equal to the number of odd columns in the Standard Young Tableau $P(p)$ into which
the RS correspondence maps $p$.  Now if $p$ is 321-avoiding, then the columns of those tableaux are of length at most two, so all odd columns must be of length 1. 
Finally, note that if $p=p_1p_2\cdots p_n$ is an involution, and $p_{i_1}<p_{i_2}<\cdots <p_{i_j}$ is an increasing subsequence in $p$, then the sequence $i_1<i_2<\cdots <i_j$ is
also an increasing subsequence in $p=p^{-1} $. So, it can be shown the only way for the involution $p$ to have a ULIS is by having a ULIS consisting entirely of {\em fixed points}. 

Let $p$ be a 321-avoiding involution with a ULIS of length $k$. Then $P(p)$ has at most two rows, of length $k$ and $n-k$. In particular, $P(p)$ has $k$ columns. On the other hand, 
as we proved in the previous paragraph, $P(p)$ must have $k$ columns of length 1. Therefore, all columns of $P(p)$ must be of length 1, so $k=n$, and $p=12\cdots n$.
This proves that $i_n(321)=1$ for all $n$. 

\section{The tivial case: the pattern 123}
If $p$ avoids 123 and has a ULIS, then that ULIS has to be of length two, if the length of $p$ is more than 1. Permutations with that property are exactly the permutations with one
non-inversion. Therefore, $u_n(123)=n-1$ for $n>1$, and $u_1(123)=1$. Such a permutation is an involution if and only if that non-inversion creates two fixed points, yielding
that $i_n(123)=0$ if $n>1$ and $n$ is odd, $i_1(123)=1$, and $i_n(123)=1$ if $n$ is even. 

\vskip 1 cm 
\begin{center} {\bf Acknowledgment} 
\end{center}
We are grateful to Zachary Hamaker for a stimulating conversation.

\end{document}